\documentclass[11pt,a4paper]{article}
\usepackage[english]{babel}
\usepackage{amsfonts}
\usepackage{amsthm}
\usepackage{amsmath}
\title{Hilbert Function and Betti Numbers of Algebras with Lefschetz Property of Order $m$}
\author{ALEXANDRU CONSTANTINESCU}

\def\k#1{\ensuremath{K[x_{1} , \ldots , x_{#1}]}}
\def\HF{Hilbert function}
\def\x#1#2{\ensuremath{(x_{1} , \ldots , x_{#1})^{#2}}} 
\def\R/I{\ensuremath{R/I}}
\def\M#1#2{\ensuremath{M_{#2 , 1} ~,~ \ldots ~,~ M_{#2 ,#1}}}
\def\Mx#1#2{\ensuremath{M_{#2,1}x_{n}^{d-#2} ~,~ \ldots ~,~ M_{#2,#1}x_{n}^{d-#2}}}
\def\T#1{\ensuremath{T_{1}~,~ \ldots ~,~ T_{#1}}}
\def\Gens{\textup{Gens}}
\def\Gin{\textup{Gin}}
\def\Lex{\textup{Lex}}
\begin{document}
\maketitle
\begin{abstract}
The authors T.Harima, J.C.Migliore, U.Nagel and J.Watanabe characterize in \cite{migliore} the \HF~of algbebras with the Lefschetz property. We extend this characterization to algebras with the Lefschetz property  $m$ times. 
We also give upper bounds for the Betti numbers of Artinian algebras with a given \HF~ and with the Lefschetz property $m$ times and describe the cases in which these bounds are reached.
\end{abstract}
\section{Introduction}

Let  $K$  be an infinite field of characteristic 0 and let  $A = \bigoplus_{d\geq 0} A_{d}$ be a homogeneous $K$-algebra, that is an algebra of the form $R/I$, where  $R$  is the polynomial ring in $n$ variables \k{n} and $I$ is a homogeneous ideal. 
We will denote  $h_{d}(A) = \dim_{K}(A_{d})$  and by 
$h(A)$ the \HF ~of $A$.

\newtheorem{def1}{Definition}[section]
\begin{def1}  

We say that an Artinian algebra $A$ has the \emph{weak Lefschetz Property (WLP) } if  there exists $\ell~\in~A_{1}$ such that  the  multiplication $\times\ell$~:~$A_{d} {\longrightarrow}~A_{d+1}$ has maximal rank  for every $d \geq1$. \\
Such an element $\ell$ is called a \emph{weak Lefschetz Element (WLE)} for $A$.

We say that $A$ has  \emph{$m$-times the weak Lefschetz Property} $(m \in \mathbb{N})$ if  there exist $\ell_{1},~\ldots~,~\ell_{m}~\in~A_{1}$ such that $\ell_{1}$ is a \emph{WLE} for A and  $\ell_{i} $ is a \emph{WLE} for $A/(\ell_{1},\ldots,\ell_{i-1}) , \forall~ i \in 2,\ldots ,m$. 
\end{def1}

The following definition uses the notion of O-sequence, which for us will just mean a sequence of natural numbers that can be the Hilbert function of some  graded  $K$-algebra. For more details on O-sequences see \cite{stan} or  \mbox{\cite[Chapter 4.2]{bh}.}

\newtheorem{def2}[def1]{Definition}
\begin{def2}

Let $ h~:~1 = h_{0},~ h_{1},~ \ldots,~ h_{s} $ be a finite O-sequence.\\
We say that $h$ is a \emph{weak Lefschetz O-sequence} if : 
\begin{itemize}
\item[-] $h$ is \emph{unimodal} $($i.e. $h_{0} < h_{1} < \ldots < h_{k} \geq h_{k+1} \geq \ldots  \geq h_{s}$ for some $k \in 0, \dots, s
)$.
\item[-] the sequence $1, h_{1} - h_{0}, \ldots , h_{k} - h_{k-1}$ is again an O-sequence.
\end{itemize}
Inductively, we say that $h$ is a \emph{$m$-times weak Lefschetz O-sequence} if:
\begin{itemize}
\item[-] $h$ is unimodal .
\item[-] the sequence $1, h_{1} - h_{0}, \ldots , h_{k} - h_{k-1}$ is  a $(m-1)$-times weak Lefschetz O-sequence.
\end{itemize}
\end{def2}

The WLP is an important property of Artinian algebras and it has been recently studied by several authors. The $m$-times WLP is  just a very natural generalization of it. For an overview  of the main results achieved  so far regarding this topic see \cite{migliore}, \cite{migliore2}. One  interesting problem is the description of the Hilbert function of Artinian algebras having the WLP. In \cite{migliore} the authors give a complete characterization of these Hilbert functions. First they make the remark that if and Artinian algebra has the WLP, then its Hilbert function must be a weak Lefschetz O-sequence in the sense of definition 1.2. and then they construct an Artinian algebra with the WLP for each  weak Lefschetz O-sequence.

In this paper we extend this characterization to Artinian algebras with $m$-times the WLP and we construct, in a more algebraic fashion, an algebra for each $m$-times weak Lefschetz O-sequence. We also answer a few natural questions regarding the Betti numbers of these algebras.

At first we will construct, using induction on $m$, and algebra that will have $m$-times the WLP. To do this, we need to start with a strongly stable ideal of the polynomial ring in one variable less than we actually need. For the case $m = 1$ we will start from the  lex-segment ideal, but the choice of the lex-segment ideal is made only in order to obtain maximal Betti numbers within the class.

The proof of the fact that the algebra we construct has $m$-times the WLP is based on a slight generalization of the description given by A.Wiebe in \cite{wiebe} of the Artinian algebras with the WLP which are the quotients of the polynomial ring by a strongly stable ideal. 

In Section 4 we show first that the algebra we construct has maximal Betti numbers among algebras with a given \HF~ and $m$-times the WLP. For this the choice of the lex-segment ideal is needed, but again it is not the only way to obtain such an algebra. In the second part of this section we give a complete description of the Artinian algebras with given \HF, $m$-times the WLP and maximal Betti numbers within this class. 
The last part of this section is dedicated to the rigidity property of these algebras. More precisely, if the upper bound is reached by the $i$-th Betti number, then it is reached also by the $k$-th Betti number, for all $k > i$.

In the fifth section we show that, by slightly  modifying our construction, we can obtain an ideal whose components of low degree define a radical ideal. To do this we will use some particular type of distraction matrix.

The results and examples presented in this paper have been inspired and suggested by computations performed by the computer algebra system CoCoA.\\

The author wishes to thank his advisor, Prof.~Aldo Conca, for his encouragement, for suggesting this problem and for his very helpful remarks on preliminary versions of this paper.

After this paper was written we were informed  that  Tadahito Harima and Akihito Wachi  have  independently  obtained some of the  results on this paper  by different methods (see \cite{harima}).



\section{Preliminaries}

If $I \subset R$ is a monomial ideal, the minimal monomial generating set of $I$ will be denoted by $\Gens(I)$. We will use also the following notion: 
\newtheorem{def3}{Definition}[section] 
\begin{def3}
A monomial ideal $I \subset R$ is called \emph{strongly stable} if:\\
For each monomial $M \in I$ and for each variable $x_{k}$ that divides $M$, we have $(\frac{x_{i}}{x_{k}})M~\in~I, ~\forall ~i < k$.
\end{def3}
It is easy to see that in order to verify if a monomial ideal is strongly stable, it is enough to verify the condition above only for the monomials in $\Gens(I)$.

To a homogeneous ideal $I \subset R$ one can canonically attach the \emph{generic initial ideal of $I$}, $ \Gin(I)$ with respect to the rev-lex order. By definition $\Gin(I)$ is the initial ideal of $I$ with respect to the rev-lex order after performing a generic change of coordinates.

For a homogeneous ideal $I \subset R$ one can also define the \emph{lex-segment ideal associated to $I$} as follows. In general, for a vector space $V$ of forms of degree $d$, one defines the vector space $\Lex(V)$ to be the vector space generated by the largest $\dim(V)$ forms in lexicographic order. Then, one defines $\Lex(I) :=  \bigoplus_{d} \Lex(I_d)$. Macaulay's theorem on \HF s (see for instance \cite{valla}) guarantees that $\Lex(I)$ will actually be an ideal, not just a graded vector space. One can immediately notice that the construction of $\Lex(I)$ depends only on the \HF~ of $R/I$, so for an O-sequence $h$ we will denote  by $\Lex(h)$ the lex-segment ideal for which the \HF~ of $R/\Lex(h)$ is $h$.

These ideals play a fundamental role in the investigation of many algebraic, homological, combinatorial and geometric properties of $I$ itself. We recall here some of their  properties  that we will use later.\\

When passing from an ideal to its generic initial ideal, the Hilbert function does not change. An important property of the generic initial ideal is that, in characteristic zero, it is strongly stable. The following result shows how the weak Lefschetz property is reflected by the generic initial ideal:
\newtheorem{thm01}[def3]{Proposition}
\begin{thm01} Let $I \subset R$ be an ideal such that $R/I$ is an Artinian algebra. Then $R/I$ has $m$-times the weak Lefschetz property if and only if $R/\Gin(I)$ has $m$-times the weak Lefschetz property. 
\end{thm01} 
 
 This result can be found for $m = 1$ in \cite{wiebe}. To see that it holds for $m>1$ one just has to follow the same proof and use  \cite[Lemma 2.1]{co1} for \mbox{$m$ linear forms.}\\

The graded Betti numbers of $I$, $\Gin(I)$ and $\Lex(I)$ satisfy the following inequalities (for details, see \cite{co2}):

\newtheorem{thm02}[def3]{Theorem}
\begin{thm02}
\begin{enumerate}
\item[\emph{(a)}]
$\beta_{ij}(R/I) \le \beta_{ij}(R/\Gin(I))$  $\forall ~i, j;$
\item[\emph{(b)}]
$\beta_{ij}(R/I) \le \beta_{ij}(R/\Lex(I))$  $\forall ~i, j.$
\end{enumerate}
\end{thm02}

Recall that a homogeneous ideal $I$ is said to be \emph{componentwise linear} if for all $k \in \mathbb{N}$ the ideal $I_{<k>}$ generated by the elements of degree $k$ in $I$ has a linear resolution.

Also, $I$ is said to be a \emph{Gotzmann} ideal if for all $k \in \mathbb{N}$ the space $I_{k}$ of forms of degree $k$ in $I$ has the smallest possible span in the next degree according to the Macaulay inequality (see \cite[Theorem 3.1]{valla}), that is  $\dim_{K}R_{1}I_{k} = \dim_{K}R_{1}\Lex(I_{k})$.

 Aramova, Herzog and Hibi characterized in \cite{bettigin} the ideals that have the same Betti numbers as their generic initial ideal as follows:

\newtheorem{thm03}[def3]{Theorem}
\begin{thm03}
The following conditions are equivalent:
\begin{enumerate}
\item[\emph{(a)}]
$\beta_{ij}(R/I) = \beta_{ij}(R/\Gin(I))$ , $\forall~ i , j;$ 
\item[\emph{(b)}]
$I$ is componentwise linear.
\end{enumerate}
\end{thm03}
 
Ideals with the same Betti numbers as the lex-segment ideal were characterized by Herzog and Hibi in \cite{bettilex}:

\newtheorem{thm04}[def3]{Theorem}
\begin{thm04}
The following conditions are equivalent:
\begin{enumerate}
\item[\emph{(a)}]
$\beta_{ij}(R/I) =\beta_{ij}(R/\Lex(I))$ , $\forall~ i , j;$ 
\item[\emph{(b)}]
$I$ is a Gotzmann ideal.\\
\end{enumerate}
\end{thm04}

The fact that if an Artinian algebra has $m$-times the WLP, then its \HF ~ is a $m$-times weak Lefschetz O-sequence follows immediately from \cite[Remark 3.3]{migliore}. The unimodality of the \HF ~is a consequence of the natural grading of the algebra. This guarantees  that if $\times \ell_{1}: A_{j} \longrightarrow A_{j+1}$ is surjective  then $\times \ell_{1}: A_{d} \longrightarrow A_{d+1}$ is surjective $\forall ~d\ge j$.
The second part of the definition of a $m$-times weak Lefschetz O-sequence is guaranteed by the fact that $A/(\ell)$ in an algebra with $(m-1)$-times the WLP and with \HF ~ $1, h_{1} - h_{0}, \ldots , h_{k} - h_{k-1}$.
\section{The construction of $R/\mathcal{W}_{m}(h)$}

Fix $h : 1=h_{0} < h_{1} < \ldots < h_{k} \ge h_{k+1} \ge \ldots \ge h_{s}$ a $m$-times weak Lefschetz O-sequence.  We will denote by $\Delta h$ the $(m-1)$-times weak Lefschetz \mbox{O-sequence:} $1,~ h_{1} - h_{0},~ \ldots ,~ h_{k} - h_{k-1}$. Inductively we will denote $\Delta^{1}h = \Delta h$ and  by $\Delta^{i}h$ the $(m-i)$-times  weak Lefschetz O-sequence given by $\Delta(\Delta^{i-1}h)$ for $i = 1, \ldots, m$.

For every finite O-sequence $h_{0}, h_{1}, \ldots , h_{s} , 0 , 0 , \ldots $ with $h_{s} \neq 0$ we will say that the length of $h$ is $s$.
Returning to our $m$-times weak Lefschetz O-sequence we will denote by $k_{i}$ the length of $\Delta^{i}h$ for every $i = 1, \ldots , m$.  Notice that $k = k_{1} \ge k_{2} \ge \ldots \ge k_{m}$. \\    

We will construct an ideal $\mathcal{W}_{m}(h)$ of $R$ such that $R/\mathcal{W}_{m}(h)$ will be the algebra we are looking for. We will first construct $\mathcal{W}_{1}(h)$, and then use induction to construct $\mathcal{W}_{m}(h)$ in the general case. 

\subsection{The case $m$=1}

Let $n = h_1$ and consider $I_{0}$ to be the lex-segment ideal of $R' = \k{n-1} $ with \HF ~ $\Delta h$. 
Now we define $I_{1}$ to be the ideal $I_{0} R$  of $R$ . It is easy to see that the \HF~ of $R/I_{1}$ is: 
\begin{displaymath}
1 = h_{0},~ h_{1},~ \ldots,~ h_{k-1},~ h_{k},~  h_{k},~ \ldots,~ h_{k},~ \ldots 
\end{displaymath}

Also, as \x{n-1}{k+1} $\subseteq I_{0}$ , we have that \x{n-1}{k+1} $\subseteq I_{1}$. So we know that all the monomials of degree $\ge k+1$ in $R$ that are not in $I_{1}$ are divisible by $x_{n}$.

In every degree $d$ we will arrange the monomials of $R$ which are not in $I_{1}$ in decreasing rev-lex order. Then we will add to the generators of $I_{1}$ the largest monomials in each degree such that we obtain the right \HF. But we first have to check how the \HF~ changes at each step in order to guarantee that this construction can be done.\\

Let $d_{0}$ be the lowest degree in which the \HF ~of $R/I_{1}$ differs from $h$. As this happens in degree higher than $k$, we know by the unimodality of $h$ that $h_{d_{0}}(R/I_{1}) > h_{d_{0}}$. So there are "too many"  monomials of degree $d_{0}$ that are not in $I_{1}$. 

We define  \mbox{$r_{0} = h_{d_{0}}(R/I_{1}) - h_{d_{0}}$}. Let $\T{r_{0}}$  be the largest (in  rev-lex order) $r_{0}$  monomials of degree $d_{0}$ not in $I_{1}$. Now we define:
\begin{displaymath}
I_{2} := I_{1} + (\T{r_{0}}).
\end{displaymath}
We want to show that the \HF of $R/I_{2}$ is: 
\begin{displaymath}
1 = h_{0},~ h_{1},~ \ldots,~ h_{d_{0}-1},~ h_{d_{0}},~  h_{d_{0}},~ \ldots,~ h_{d_{0}},~ \ldots 
\end{displaymath}

Obviously the \HF~ of $R/I_{2}$ is equal to the one of $R/I_{1}$ in degree smaller than $d_{0}$ and now also in degree $d_{0}$ it is exactly $h_{d_{0}}$.

Denote by \M{u_{i}}{i} $\in R'$ the  monomials of  degree $i$ which are not in the original $I_{0}$ ($u_{i}$ will be equal to $h_{i}-h_{i-1}$). These will be  the  monomials  on degree $\le k$ in the first $(n-1)$ variables that are not in $I_{1}$. So the monomials of degree $d > k$ that are not in $I_{1}$ are the following:
\begin{displaymath}
\Mx{u_{k}}{k}~ ,~ \ldots ~,~ \Mx{u_{1}}{1} 
\end{displaymath}

We have $T_{1} = M_{k,1}x_{n}^{d_{0}-k}$, and let $i_{0}$ and $j_{0}$ be the index for which $T_{r_{0}}~=~M_{i_{0},j_{0}}x_{n}^{d_{0} - i_{0}}$ (the $r_{0}$-th largest monomial of degree $d_{0}$ not in  $I_{1}$ ).  After adding to $I_{1}$ these first  $r_{0}$ monomials,   we get that $\dim((\R/I_{2})_{d}) \le h_{d_{0}}$ for $d > d_{0}$.\\
 
 Suppose  there exists a  monomial of degree $d > d_{0}$, $M_{t,r}x_{n}^{d-t} \notin I_{1}$, with  ($t < i_{0}$) or ($t=i_{0}$~and~~$r > j_{0}$), i.e. that is not in those first $r_{0}$ monomials added to $I_{1}$, but $M_{t,r}x_{n}^{d-t} \in I_{2}$ .  As $M_{t,r} \notin I_{0} ,~~ M_{t,r}x_{n}^{d-t} $ must be divisible by a generator of $I_{2}$, who itself is divisible by $x_{n}$.  So it must be divisible by one of $M_{k,1}x_{n}^{d_{0} - k},~ \ldots ,~ M_{i_{0},j_{0}}x_{n}^{d_{0}-i_{0}}$.  

Let $M_{i,j}x_{n}^{d_{0}-i}$ be that monomial. It follows that $M_{i,j} | M_{t,r}$, so  $i \le t$. As $ i \ge i_{0} \ge t$ it follows that $i = t$.  So they have the same degree, but $ r > j_{0} \ge j$ so they are different and the divisibility can not take place - a contradiction.\\
So the only monomials that belong to $I_{2}$ but not to $I_{1}$ in degree $d \ge d_{0}$ are exactly $M_{k,1}x_{n}^{d - k},~ \ldots ,~ M_{i_{0},j_{0}}x_{n}^{d -i_{0}}$.\\

So we have shown that after adding the necessary monomials to $I_1$ in the first degree where this is needed ($d_{0}$), the \HF~ of the new algebra $R/I_{2}$ will become: 
\begin{displaymath}
1 = h_{0},~ h_{1},~ \ldots,~ h_{d_{0}-1},~ h_{d_{0}},~  h_{d_{0}},~ \ldots,~ h_{d_{0}},~ \ldots 
\end{displaymath}

This procedure can be repeated as from degree $> k$ the original weak Lefschetz O-sequence is decreasing, and after a finite number of steps (at most $s - k$) we will obtain a new ideal, which we will denote by $\mathcal{W}_{1}(h)$, such that $R/\mathcal{W}_{1}(h)$ has the desired Hilbert function. \\

So we have constructed a monomial  ideal with \HF~ $h$ and with the property that $\x{n-1}{k+1} \subseteq \mathcal{W}_{1}(h)$. We also have that all the generators which are divisible by $x_{n}$ appear in degree $\ge k+1$ and that the generators not divisible by $x_{n}$ appear  in degree $\le k+1$.

In order to be able to apply induction we will need to prove the following:

\newtheorem{strstab}{Lemma}[section]
\begin{strstab}
The ideal $\mathcal{W}_{1}(h)$ is strongly stable.
\end{strstab}

\begin{proof}
By construction $\mathcal{W}_{1}(h)$ is a monomial ideal. Let $M \in \Gens(\mathcal{W}_{1}(h))$ be a monomial of degree $d$. We want  to prove that $x_i\frac{M}{x_j} \in \mathcal{W}_{1}(h),  \forall ~j$ such that $x_j | M$ and $\forall~ i < j$.

  We distinguish two cases:\\
1. If $x_{n} \not| M$, then $M$ could be seen as a monomial of $I_{0}$ which is the lex-segment ideal for $\Delta h$. As the lex-segment ideal is strongly stable, we get that $x_{i} \frac{M}{x_{j}} \in I_{0} , \forall~ j$ such that $x_{j} | M$ and $ \forall~ i < j$.\\
2. If $x_{n} | M$, then let $j \in 1,\ldots, n$ be such that $x_{j} | M$ and let $i < j$.  Then we have $x_{i} \frac{M}{x_{j}} \ge_{rev-lex} M  $ and so we must have that $x_{i} \frac{M}{x_{j}} \in \mathcal{W}_{1}(h)$  by construction, because we chose as generators the largest monomials in \mbox{rev-lex} order. \\
\end{proof}

\subsection{The general case}

Let $m \in \mathbb{N}$, $m \ge 2$. Assume we can construct an algebra $R'/\mathcal{W}_{m-1}(\Delta h) $, with \HF ~$\Delta h$,  such that $\mathcal{W}_{m-1}(\Delta h)$ is a strongly stable ideal of  $R' = \k{n-1}$ and that $\x{n-i}{k_{i} +1} \subseteq \mathcal{W}_{m-1}(\Delta h)$ for all $i = 2, \ldots, m-1$.\\

Now we define $I_1 = \mathcal{W}_{m-1}(\Delta h) R$. The \HF~ of  $R/I_1$ will be $1~=~h_{0},~h_{1},~\ldots,~h_{k-1},~h_{k},~h_{k},~\ldots,~h_{k},~\ldots$  and following the method of adding the needed highest monomials in rev-lex order as in the case of $m = 1$, we can  construct an ideal $\mathcal{W}_{m}(h)$. The same arguments as in the case $m = 1$ prove that the construction can be done,
 because in that case  we didn't use the fact that  $I_{0}$ was the lex-segment ideal, we just used the fact that it was a strongly stable ideal. The choice of $I_{0}$ as the lex-segment ideal is needed for obtaining maximal Betti numbers.
 
In fact, since $\mathcal{W}_{m-1}(\Delta h)$ is strongly stable, the  proof of Lemma 2.1 works also for proving that $\mathcal{W}_{m}(h)$ is strongly stable.

\subsection{$R/\mathcal{W}_{m}(h)$ has $m$-times the WLP}

In this section we will show that the algebra we have constructed so far is actually what we wanted:

\newtheorem{WLP1}[strstab]{Proposition}
\begin{WLP1}
$R/\mathcal{W}_{m}(h)$ has $m$-times the weak Lefschetz Property.
\end{WLP1}

In order to prove this, we will use the following result from \cite{wiebe}:
\theoremstyle{plain} \newtheorem{lema1}[strstab]{Lemma}
\begin{lema1}
If $I$ is a strongly stable ideal of $R = \k{n}$ then:\\
\R/I has the WLP $\Longleftrightarrow x_{n}$ is a WLE for \R/I.
\end{lema1}

From this result we can deduce the following one:

\newtheorem{lema3}[strstab]{Lemma}
\begin{lema3}
 If $I$ is a strongly stable ideal of $R = \k{n}$, then the following are equivalent.
 \begin{enumerate}
 \item[\emph{1.}]
 \R/I has the WLP.
\item[\emph{2.}]
\begin{itemize} 
\item[\textup{(a)}] $h(R/I)$ is unimodal :
$h_{0}<h_{1}<\ldots<h_{k}\geq h_{k+1}\geq\ldots\geq h_{s}$, 
\item[\textup{(b)}] 
\x{n-1}{k+1} $\subseteq I$, 
\item[\textup{(c)}] If $M \in \Gens(I)$ is divisible by $x_{n}$, then  $\deg(M) \ge k+1$.
\end{itemize}
\end{enumerate}
\end{lema3}
 
 \begin{proof} 
 1.$\Rightarrow$2. The fact the \HF~ is unimodal is already known from \cite{migliore}.\\ By Lemma 3.3 we know that $x_{n}$ is a WLE for \R/I,  so the multiplication  $\times x_{n}: (R/I)_{d} \rightarrow (R/I)_{d+1}$ must be of maximal rank, i.e. injective if $d < k$ and surjective if $d \ge k$. This implies immediately that \x{n-1}{d} $\subseteq I$ for $d >k$.\\
 Suppose that there is a minimal generator $M$ of $I$, which has degree $d < k+1$, and $x_{n} | M$. Then $\frac{M}{x_{n}} \neq 0$ in $(R/I)_{d-1}$ but is taken by the multiplication with $x_{n}$ to $M = 0$ in $(R/I)_{d}$ - a contradiction with the injectivity of $\times x_{n}$.
 
 2.$\Rightarrow$1. 
 We will show that $x_{n}$  is a WLE for \R/I.  As we have that \x{n-1}{k+1} $\subseteq I$,  it follows that the multiplication by $x_{n}$ is surjective in degree $\ge k$. \\
Let $ d < k$ suppose that there exists a monomial of degree $d$, $M  \in R$ and $M \notin I_d$,  such that $x_{n}M \in I_{d+1}$. This means that $x_{n}M$ is divisible by a minimal generator $G$ of $I$. As $\deg(G) \le k$, we have that $x_{n} \not| G$. This means that $G | M$ contradicting the fact that $M \notin I_d$. 
 \end{proof}

Let us notice that $2. \Rightarrow 1.$ of Proposition 3.4 holds also when $I$ is just a monomial ideal, not necessarily a strongly stable one.

Now we can prove proposition 3.2:
\begin{proof}
We will use induction on $m$:
If $m = 1$ we can see very easy that the conditions (a), (b) and (c) from Lemma 3.3 are satisfied by construction, so as  $\mathcal{W}_{1}(h)$ is strongly stable we can apply Lemma 3.3 and get that $R/\mathcal{W}_{1}(h)$ has the WLP.

Suppose that the proposition is  true for $m-1$. This means that the algebra $R'/\mathcal{W}_{m-1}(\Delta h)$ has $(m-1)$-times the WLP ($R' = \k{n-1}$).  As $\mathcal{W}_{m}(h)$ is strongly stable and again the conditions of Lemma 3.3 are satisfied, we get that $R/\mathcal{W}_{m}(h)$ has the WLP and, by Lemma 3.2, $x_{n}$ is a WLE. As by construction $R/\mathcal{W}_{m}(h)+(x_{n}) = R'/\mathcal{W}_{m-1}(\Delta h)$ which has by hypothesis $(m-1)$-times the WLP, we get that $R/\mathcal{W}_{m}(h)$ has $m$-times the WLP.
\end{proof}





\section{Ideals with Maximal Betti Numbers}

In this section we will first show that $R/\mathcal{W}_{m}(h)$ has maximal Betti numbers among algebras with \HF~ $h$ and $m$-times the WLP. Then we will characterize all other ideals that have maximal Betti numbers within this class. In the third part of this section we will show that these upper bounds are rigid.  

\subsection{ $R/ \mathcal{W}_{m}(h)$ has maximal Betti numbers}

We want to prove the following:
\newtheorem{maxbetti}{Proposition}[section]
\begin{maxbetti}
For any  algebra $R/J$ that has \HF~ $h$ and $m$-times the WLP we have:
\begin{equation} \label{eq:beta}
\beta_{ij}(R/J) \leq \beta_{ij}(R/ \mathcal{W}_{m}(h))~,~~~~~~    \forall i , j \ge 0
\end{equation}
\end{maxbetti}

We have seen in Section 2 that for a homogeneous ideal $J \subset R$ taking its generic initial ideal $\Gin(J)$ does not change the \HF, and also that $R/J$ has $m$-times the WLP if and only if $R/\Gin(J)$ has $m$-times the WLP. From Theorem 2.3 we have the following inequality: \[ \beta_{ij}(R/J) \leq \beta_{ij}(R/\Gin(J)) ,~~\forall  i, j \ge 0. \]
So, as $\Gin(J)$ is a strongly stable ideal, it will be enough to prove that  (\ref{eq:beta}) holds for $J$ strongly stable.  \\                                                                                                                                                                                                                                                                                      

First let us establish some notations. For a monomial $M = x_{1}^{a_{1}}\ldots x_{n}^{a_{n}}$ in \k{n} we define: 
\[\max(M) = \max\{ i : a_{i} > 0 \}. \] 

For a set of monomials $A \subset$  \k{n} and for $i = 1 , \dots , n $ we write:
\begin{displaymath}
m_{i}(A) = | \{ M \in A : max(M) = i\}| , \quad
m_{\le i}(A)= | \{ M \in A : max(M) \le i \} |.
\end{displaymath}
When $J$ is either a vector space generated by monomials of the same degree or a monomial ideal, we set \[ m_{i}(J) = m_{i}(G), \qquad m_{\le i}(J) = m_{\le i}(G) ,\] where G is the set of minimal monomial (vector space or ideal) generators of $J$. If $J$ is a monomial ideal we will denote by $J_{i}$ the vector space \mbox{$ \{ M \in J : deg(M) = i \} $.} 

We will need the following result from \cite{co2}: 
\newtheorem{prop1}[maxbetti]{Proposition}
\begin{prop1} Let $ I, J$ be strongly stable ideals with the same \HF. Assume that $m_{\le i}(I_{j}) \le m_{\le i}(J_{j}) ,~~  \forall i , j \ge 0$. Then one has: 
\begin{enumerate}
\item[\emph{1.}] $ m_{i}(J) \le m_{i}(I), \qquad ~~~~ \qquad \forall ~i > 0 .$
\item[\emph{2.}] $\beta_{ij}(R/J) \leq \beta_{ij}(R/I) , \qquad  \forall~ i , j \ge 0 .$
\end{enumerate}
\end{prop1}

We can now prove Proposition 4.1:
\begin{proof}
We saw that we can suppose that $J$ is strongly stable and, as ~$\mathcal{W}_{m}(h)$ is also strongly stable, from proposition 4.2 we have that in order to prove that (\ref{eq:beta}) holds, we only need to prove that:
\begin{equation} \label{eqm}
m_{\le i}( (\mathcal{W}_{m}(h))_{j}) \le m_{\le i}(J_{j}) ,~~  \forall~ i \le n \textup{~and~}\forall ~j \ge 0.
\end{equation}

As $R/J$ and $R/ \mathcal{W}_{m}(h)$ have the same \HF~ it follows immediately that (\ref{eqm}) holds for $i = n$.
 
If $i < n$ it is easy to see that, as $\mathcal{W}_{m-1}(\Delta h) =  (\mathcal{W}_{m}(h) +(x_{n}))/(x_n)$ and  if we denote \mbox{$J_{m-1} = (J + (x_{n}))/(x_n)$,} then we have:
\[ m_{\le i}(( \mathcal{W}_{m}(h))_{j}) = m_{\le i}(( \mathcal{W}_{m-1}(\Delta h))_{j}) ~~ \forall~ i < n~~ \textup{and} \]
\[ m_{\le i}(J_{j}) = m_{\le i}((J_{m-1})_{j}) ~~~\qquad ~\forall~i < n.\] 
So what we have to prove now is that
\begin{equation} \label{eqm2}
m_{\le i}( (\mathcal{W}_{m-1}(\Delta h))_{j}) \le m_{\le i}((J_{m-1})_{j}) ,~~  \forall~ i < n \textup{~and~}\forall ~j \ge 0.
\end{equation}
This means that if (\ref{eqm}) holds for $m-1$, then it also holds for $m$. So, in order to conclude, we only need to look at the case $m = 1$.

If $m = 1$ we have:

1. If $j > k_{1}$ we have by Lemma 3.3 that 
\[ \x{n-1}{j} \subseteq  \mathcal{W}_{1}(h) \quad and \quad \x{n-1}{j} \subseteq J \] 
So, in this case,  we actually have equality in (\ref{eqm2})  $\forall ~i < n$.

2. If $j \le k_{1}$  
By construction $ \mathcal{W}_{0}(\Delta h)$ is  the lex-segment ideal and $J_{0}$ is still a strongly stable ideal (see \cite[Proposition 1.4]{big}).  By Lemma 3.2, $x_{n}$ is a WLE for both $R/J$ and $R/\mathcal{W}_{m}(h)$, and thus we have that the \HF s of $R'/J_{0}$ and $R'/\mathcal{W}_{0}(\Delta h)$ are equal to $\Delta h$.

From the equality of the Hilbert functions  we have   $|( \mathcal{W}_{0}(\Delta h))_{j}| = |(J_{0})_{j}|$ and thus we can apply a result of A.M. Bigatti (see \cite[theorem 2.1]{big}) that ensures that (\ref{eqm}) holds also for $m = 1$.

\end{proof}

\subsection{Other Ideals with Maximal Betti Numbers}

To simplify notation we introduce, for all $i \in 1, \ldots n$ the following morphism: 
$ \rho_i : \k{n} \longrightarrow \k{i}$, 
with: 
\begin{displaymath}
\rho_i(x_j)  = \left.\bigg\{ \begin{array}{ll}
                                   x_j &\textup{if}~ j \le i \\
                                   0    &\textup{if} ~ j > i.\\
                                   \end{array}\right.
\end{displaymath}                                 
Notice that if $I \subset \k{n}$ is a homogeneous ideal, then $\rho_i(I)$  is an ideal of \k{i}. This ideal will have the same generators as the ideal
 \mbox{$I + (x_n, \ldots, x_{n-i+1}) / (x_n, \ldots, x_{n-i+1})$.}\\

In this section we will give a description of the ideals $J$ of $R$ such that $R/J$ has \HF~ $h$,  $m$-times the WLP and maximal Betti numbers within this category. More precisely we will prove the following:
\newtheorem{bettimax}[maxbetti]{Proposition}
\begin{bettimax}
Let $J \subset R$ be an ideal such that $R/J$ has \HF~ $h$ and $m$-times the weak Lefschetz property $(m \in \mathbb{N})$.  The following are equivalent:
\begin{enumerate}
\item[\emph{1.}]
$J$ has maximal Betti numbers  among ideals with the above properties.
\item[\emph{2.}]
$J$ is componentwise linear and the ideal $\rho_{n-m}(\Gin(J))$ is Gotzmann.
\end{enumerate}
\end{bettimax}
 
We have already seen in Proposition 4.1 that $\mathcal{W}_{m}(h)$ has maximal Betti numbers. Let us 
fix $J \subset R$ as in the hypothesis of Proposition 4.3. From Theorem 2.3 and Proposition 4.1  we get:   
\begin{displaymath}
\beta_{ij}(R/J) \le \beta_{ij}(R/\Gin(J)) \le \beta_{ij}(R/\mathcal{W}_{m}(h))
\end{displaymath}

This means that if $R/J$ has maximal Betti numbers among algebras with $m$-times the WLP and \HF~ $h$, then \mbox{$\beta_{ij}(R/J) = \beta_{ij}(R/\Gin(J))$.}
In other words, $J$ must be componentwise linear (by Theorem 2.4.).

Knowing this, we will now concentrate on the properties of $\Gin(J)$. Replacing $J$ with $\Gin(J)$ we may assume that $J$ is strongly stable.

For a homogeneous ideal $J \subseteq R$  and $i \in \mathbb{N}$ we will denote by $J_{\le i}$  the ideal generated by the elements of $J$ with degree $\le i$. If $J$ is monomial, then $J_{\le i}$ will also be monomial.

We already know from Lemma 3.2 that  for a strongly stable ideal $J$, $R/J$ has the WLP if and only if $x_{n}$ is a WLE for $R/J$. An easy generalization of this fact is the following:
\newtheorem{lema4}[maxbetti]{Lemma}
\begin{lema4}
Let $J \subseteq R$ be a strongly stable ideal. Then\\
 $R/J$ has \mbox{$m$-times} the WLP  $\Longleftrightarrow$  $x_{n-i}$ is a WLE for $R/J + (x_n, \ldots, x_{n-i+1})$ for all \mbox{$i = 0,\ldots, m-1$.}
\end{lema4}
\begin{proof} When  $m=1$ the result is just the one of Lemma 3.2.\\
If $m>1$ then still we know that $x_{n}$ is a WLE for $R/J$. But  $R/J +(x_{n})$ has $(m-1)$-times the WLP and $J + (x_n)/(x_{n})$ will be still strongly stable so we can apply induction.
\end{proof}

We will prove the following result which, together with the above observations, proves Proposition 4.3.

\newtheorem{prop2}[maxbetti]{Proposition}
\begin{prop2}
Let $J \subset R$ be a strongly stable ideal  such that $R/J$ has $m$-times the WLP and \HF~$h$. Then
\begin{center}
$\beta_{ij}(R/J) = \beta_{ij}(R/\mathcal{W}_{m}(h)), \forall~i,j \Longleftrightarrow \rho_{n-m}(J_{\le k_{m}})$ is Gotzmann.
\end{center}
\end{prop2}

\begin{proof}
 The ideals $J$ and $\mathcal{W}_{m}(h)$ are strongly stable with the same \HF . We know that $m_{\le i}(\mathcal{W}_{m}(h)_{j}) \le m_{\le i}(J_{j})~~ \forall~ i,j$ from the proof of Proposition 4.1. From \cite[Proposition 3.7]{co2} we know that the following are equivalent:
\begin{equation}\label{betti}
 \beta_{ij}(R/J) = \beta_{ij}(R/\mathcal{W}_{m}(h)),~\forall~ i,j.
\end{equation}
\begin{equation}\label{metti}
 m_{\le i}(J_{j}) = m_{\le i}((\mathcal{W}_{m}(h))_{j}),~\forall~ i,j.
\end{equation}
 
 \fbox{$\Rightarrow$} So we have  $m_{\le i}(J_{j}) = m_{\le i}((\mathcal{W}_{m}(h))_{j}),~\forall~ i,j$, but this means also that 
\begin{displaymath}
 m_{\le i }((\rho_{n-m}(J))_{j}) = m_{\le i}((\rho_{n-m}(\mathcal{W}_{m}(h)))_{j})  ~~ \forall~ i,j.
 \end{displaymath}
 
By construction $\rho_{n-m}(\mathcal{W}_{m}(h)) = \mathcal{W}_{0}(\Delta^{m}h) = \Lex(\Delta^{m}h)$, which  is a strongly stable ideal. Also $\rho_{n-m}(J)$  is a strongly stable ideal because its generators are just the generators of $J$ in the first $n-m$ variables.
  
 By Lemma 4.4 we have that $x_{n-i}$ is a WLE for  $R/\rho_{n-i}(\mathcal{W}_{m}(h))$ and for $R/\rho_{n-i}(J)$, $\forall ~i = 0, \ldots, m-1$. Thus we get that both $R/\rho_{n-m}(\mathcal{W}_{m}(h))$ and $R/\rho_{n-m}(J)$ have the same \HF, $\Delta^{m} h$.
 
 So we can apply  \cite[Proposition 3.7]{co2} and obtain that 
 \[ \beta_{i,j}(\rho_{n-m}(J)) = \beta_{i,j}(\Lex(\Delta^{m}h)),\]
  which means by Theorem 2.4. that $\rho_{n-m}(J)$ is a Gotzmann ideal.
    
 \fbox{$\Leftarrow$} We will show that (\ref{metti}) holds.
 
1. If $i = n-t \ge n-m$ then (\ref{metti}) holds from the equality of the \HF s of $R/\rho_{n-t}(J)$ and $R/\rho_{n-t}(\mathcal{W}_{m}(h))$.

2. If $i < n-m$,  we have 
\begin{displaymath} 
m_{\le i}(J_{j}) =  m_{\le i}((\rho_{n-m}(J))_{j}),
\end{displaymath}
\begin{displaymath} 
m_{\le i}((\mathcal{W}_{m}(h))_{j}) = m_{\le i}((\rho_{n-m}(\mathcal{W}_{m}(h))_{j})
\end{displaymath} 
So we only need to prove (\ref{metti}) for  $\rho_{n-m}(J)$ and \mbox{$\rho_{n-m}(\mathcal{W}_{m}(h)) = \Lex(\Delta^{m}h)$.} 
In this case (\ref{metti}) holds because $\rho_{n-m}(J)$ is a Gotzmann ideal, which is equivalent by Theorem 2.5 to the equality of its Betti numbers with the Betti numbers of the lex-segment ideal. This is again equivalent by \cite[Proposition 3.7]{co2} to 
\begin{displaymath} 
m_{\le i}((\rho_{n-m}(J))_{j}) = m_{\le i}((\Lex(\Delta^{m}h))_{j}).
\end{displaymath} 

 


 \end{proof} 
\subsection{Rigid Resolutions}

For a homogenous ideal $I$ it has been shown in \cite{co3} that 
if $\beta_{q}(I) = \beta_{q}(\Gin(I))$ then $\beta_i(I) = \beta_i(\Gin(I))$  for all $i \ge q$. This property is called rigidity and it also holds if $\Gin(I)$ is replaced by $\Lex(I)$ or any generic initial ideal of $I$.

In this section we will prove that  algebras with $m$-times the WLP have a similar property:  if one of the  Betti numbers reaches the upper bound given by the Betti numbers of $R/\mathcal{W}_{m}(h))$, then all the following Betti numbers reach it as well. More precisely we will prove that:

\newtheorem{rigid}[maxbetti]{Proposition}
\begin{rigid}
Let $I \subset R$ be a homogeneous ideal such that $R/I$ has $m$-times the WLP and \HF~$h$. If $\beta_{q}(R/I) = \beta_{q}(R/\mathcal{W}_{m}(h))$ for some $q$ then $\beta_{i}(R/I) = \beta_{i}(R/\mathcal{W}_{m}(h))$ for all $i  \ge q$.
\end{rigid}

\begin{proof}
We have already seen that we have the following inequalities:
\begin{displaymath}
\beta_{i}(R/I) \le \beta_{i}(R/\Gin(I)) \le \beta_{i}(R/\mathcal{W}_{m}(h)).
\end{displaymath}
If for some $q$ equality takes place, we have from \cite[Corollary 2.4]{co3} the following: $\beta_{i}(R/I) = \beta_{i}(R/\Gin(I))$ for all $i \ge q$. So we just need to prove that the proposition holds for \Gin(I), i.e we can assume that $I$ is a strongly stable ideal. 

From the Eliahou-Kervaire formula for the Betti numbers of stable ideals (see for example \cite{big}) we have that:
\begin{equation}\label{elik}
 \beta_{i}(R/I) = \sum_{s=i}^n m_{s}(I)\binom{s-1}{i-1} 
\end{equation}
In the proof of Proposition 4.1 we have shown that the inequality (\ref{eqm}) takes place, so from Proposition 4.2 we have that:
\begin{equation}\label{mi}
m_{i}(I) \le m_{i}(\mathcal{W}_{m}(h)), \qquad \forall ~i > 0 .
\end{equation}
So by (\ref{elik}) and (\ref{mi}) we have that  $\beta_{q}(R/I) = \beta_{q}(R/\mathcal{W}_{m}(h))$  also implies the following equality:
\[ m_{i}(I) = m_{i}(\mathcal{W}_{m}(h)), \forall ~i \ge q.\] 
So, again by (\ref{elik}), we get that  $\beta_{i}(R/I) = \beta_{i}(R/\mathcal{W}_{m}(h))$ for all $i  \ge q$.
\end{proof}

\newtheorem{corrigid}[maxbetti]{Corollary}
\begin{corrigid}
Let $I \subset R$ be a homogeneous ideal such that the graded algebra $R/I$ has $m$-times the WLP and \HF~$h$. 

If $\beta_{q}(R/I) = \beta_{q}(R/\mathcal{W}_{m}(h))$ for some $q$ then:
\[\beta_{ij}(R/I) = \beta_{ij}(R/\mathcal{W}_{m}(h)) \qquad \forall~i\ge q ,~ \forall~ j.\]
\end{corrigid}
\begin{proof}
By proposition 4.1 we have $\beta_{ij}(R/I) \le \beta_{ij}(R/\mathcal{W}_{m}(h)) ~\forall~i, j,$ and as $\beta_i (R/I)= \sum_j \beta_{ij} (R/I)$, Proposition 4.6 implies the desired equality.
\end{proof}

\section{Ideal of points}

In this section we will construct, starting from $\mathcal{W}_{m}(h)$ and using a distraction matrix, another ideal $I$ (with the same \HF~and Betti numbers)  such that $R/I$ still has $m$-times the WLP and $I_{\le k_{1}}$ is the ideal of finite set of rational points in $\mathbb{P}_{K}^{n-1}$. \\

First let us recall some notions and results that we need. The results on distractions that we will present here were proven by Bigatti, Conca and Robbiano in  \cite{distract}.
\newtheorem{distrmatr}{Definition}[section]
\begin{distrmatr} 
Let $\mathcal{L} = (L_{ij}~|~i = 1,\ldots,n~ ,~ j~ \in~ \mathbb{N})$ be an infinite matrix with entries $L_{ij} \in R_{1}$ with the following properties:
\begin{enumerate}
\item[\textup{1.}] $\{L_{1j_{1}},\ldots, L_{nj_{n}}\}$ generates $R_{1}$ for every $j_{1},\ldots,j_{n} \in \mathbb{N}$.
\item[\textup{2.}] There exists an integer $N \in \mathbb{N}$ such that $L_{ij} = L_{iN}$ for every $j > N$.
\end{enumerate}
We call $\mathcal{L}$ an \emph{$N$-distraction matrix} or simply a distraction matrix.
\end{distrmatr}

\newtheorem{distr}[distrmatr]{Definition}
\begin{distr}
Let $\mathcal{L}$ be a distraction matrix, and $M = x_{1}^{a_{1}}x_{2}^{a_{2}}\ldots x_{n}^{a_{n}}$ a monomial in $R$. Then the polynomial 
$D_{\mathcal{L}}(M) = \prod_{i=1}^{n}(\prod_{j=1}^{a_{j}}L_{ij})$
is called the \emph{$\mathcal{L}$-distraction} of $M$.
\end{distr}
Having defined $D_{\mathcal{L}}(M)$ for every monomial, $D_{\mathcal{L}}$ extends to a $K$-linear map. Therefore we can consider $D_{\mathcal{L}}(V)$ where $V$ is a subvector space of $R$, and call it the \emph{$\mathcal{L}$-distraction} of $V$. 

The ideal that we will construct will be $D_{\mathcal{L}}( \mathcal{W}_{m}(h))$ for some distraction matrix $\mathcal{L}$ with some extra properties. When $I$ is a homogeneous ideal of $R$ , $D_{\mathcal{L}}(I)$ will coincide with $\bigoplus_{d}D_{\mathcal{L}}(I_{d})$, which is in general just  a vector space, not an ideal. However, when $I$ is a monomial ideal we have the following result:

\newtheorem{distrideal}[distrmatr]{Proposition}
\begin{distrideal}
Let $\mathcal{L}$ be a distraction matrix, and $I\subset R$ a monomial ideal.
\begin{enumerate}
\item[\emph{1.}] The vector space $D_{\mathcal{L}}(I)$ is a homogeneous ideal in $R$.
\item[\emph{2.}] If $M_{1},\ldots,M_{r}$ are monomials in $R$ such that $I = (M_{1},\ldots,M_{r})$, then we have the following: $D_{\mathcal{L}}(I)=(D_{\mathcal{L}}(M_{1}),\ldots,D_{\mathcal{L}}(M_{r}))$.
\item[\emph{3.}]  $h(R/I) = h(R/D_{\mathcal{L}}(I))$.
\item[\emph{4.}]  $\beta_{ij}(R/I) = \beta_{ij}(R/D_{\mathcal{L}}(I))\quad\forall~i,j.$
\end{enumerate}
\end{distrideal}

So we know now that $R/D_{\mathcal{L}}(\mathcal{W}_{m}(h))$ will still have Hilbert function $h$. But will $R/D_{\mathcal{L}}(\mathcal{W}_{m}(h))$ still have $m$-times the weak Lefschetz property?  The following result \cite[Theorem 4.3]{distract} will lead us to the answer of this question: 
\newtheorem{distrgin}[distrmatr]{Theorem}
\begin{distrgin}
Let $\mathcal{L}$ be a distraction matrix and $I \subset R$ be a strongly stable monomial ideal. Then $\Gin( D_{\mathcal{L}}(I)) = I$.
\end{distrgin}
So from Proposition 2.2 it follows that also $R/D_{\mathcal{L}}( \mathcal{W}_{m}(h))$ has $m$-times the weak Lefschetz property.  \\

We still want to show  that $D_{\mathcal{L}}( \mathcal{W}_{m}(h))_{\le k_{1}}$ 
is an ideal of a finite set of points. For this we need to recall the following notion:
 \newtheorem{raddistr}[distrmatr]{Definition}
 \begin{raddistr}
 Let $I = (x_{i_{1}}^{a_{1}},\ldots,x_{i_{r}}^{a_{r}}) \subset R$ be an irreducible monomial ideal and let $S = \{s=(s_{1},\ldots,s_{r})~|~ 1\le s_{i} \le a_{i}~ , \forall~ i = 1,\ldots,r\}$. Let $\mathcal{L}$ be a distraction matrix, and let $V_{s}$ be the $K$-vector space generated by $\{L_{i_{1}s_{1}},\ldots, L_{i_{r}s_{r}}\}$. If  $V_{s} \neq V_{s'}$, $\forall ~s, s' \in S$ $(s \neq s')$,  we say that $\mathcal{L}$ is \emph{radical for $I$}.
 
 More generally, if $I$ is any monomial ideal, we say that $\mathcal{L}$ is \emph{radical for $I$} if $\mathcal{L}$ is radical for all the irreducible components of $I$.
  \end{raddistr}
The following result will show us how we need to choose the distraction matrix $\mathcal{L}$ in order to obtain the desired construction (see \cite[Corollary 4.10]{distract}). 
 \newtheorem{points}[distrmatr]{Proposition}
 \begin{points}
 Let $I \subset \k{n-1}$ be a zero-dimensional strongly stable monomial ideal, and let $\mathcal{L}$ be a distraction matrix which is radical for $I$, and whose entries are in the polynomial ring $R= \k{n}$. 
 
 Then $D_{\mathcal{L}}(I)$ is the ideal of a finite set of points in $\mathbb{P}_{K}^{n-1}$ such that $\Gin(D_{\mathcal{L}}(I)) = IR$.
 \end{points} 
 First let us notice that by Proposition 5.3 we have:
\begin{displaymath}
 D_{\mathcal{L}}( \mathcal{W}_{m}(h))_{\le k_{1}} = (D_{\mathcal{L}}(M) ~|~ M \in \Gens(\mathcal{W}_{m}(h)) , deg(M)\le k_{1}).
 \end{displaymath}
and that by construction the generators of $(\mathcal{W}_{m}(h))_{\le k_{1}}$ are the generators of  $\mathcal{W}_{m-1}(\Delta h)$ so they are monomials in $x_{1},\ldots,x_{n-1}$.\\
We choose $\mathcal{L}$ to be a distraction matrix such that the first $(n-1)$ lines form a distraction matrix $\mathcal{L'}$ that is radical for $\mathcal{W}_{m-1}(\Delta h)$, and has entries is \k{n}. So  $D_{\mathcal{L'}}(\mathcal{W}_{m-1}(\Delta h)) = D_{\mathcal{L}}((\mathcal{W}_{m}(h))_{\le k_{1}}) = D_{\mathcal{L}}( \mathcal{W}_{m}(h))_{\le k_{1}}$. Together with the arguments presented so far in this section, this proves the following:

 \newtheorem{proppoints}[distrmatr]{Proposition}
 \begin{proppoints}
 Let $\mathcal{L}$ be a distraction matrix such that the first $n-1$ lines form a distraction matrix $\mathcal{L'}$ that is radical for $\mathcal{W}_{m-1}(\Delta h)$. Then :\\
 $R/D_{\mathcal{L}}(\mathcal{W}_{m}(h))$ has $m$-times the WLP, \HF~ $h$, the same Betti numbers as $R/\mathcal{W}_{m}(h)$
  and the ideal $D_{\mathcal{L}}( \mathcal{W}_{m}(h))_{\le k_{1}}$ is the ideal of a finite set of rational points in $\mathbb{P}_{K}^{n}$. 
 \end{proppoints}
This proposition is a  generalization of  the results obtained by T.Harima, J.C.Migliore, U.Nagel and J.Watanabe in \cite[Theorem 3.20]{migliore}.

\section{Examples}

Let $h: 1 , 4 , 7 , 8 , 7 , 4 , 1$ be our given O-sequence and let $R = K[x,y,z,t]$.  We will have $\Delta h: 1 , 3 , 3 , 1$ and $\Delta^{2}h: 1, 2$. So we see that $h$ is a $2$-times weak Lefschetz O-sequence. 

We will construct $\mathcal{W}_{2}(h)$ as well as $\mathcal{W}_{1}(h)$ and see that they are different. 

Let us first construct $\mathcal{W}_{2}(h)$. We start with the lex-segment ideal of $\Delta^{2}h $ which is the ideal 
\[\Lex(\Delta^{2}h ) = (x^{2}, xy, y^{2}) \subset K[x, y].\]
 The ideal $\Lex(\Delta^{2}h )S$, where $S = K[x, y, z]$ will have the \HF: 
 \[1, 3, 3, 3, 3, \ldots\]
 The  monomials of $S$ of degree $d > 2$ that are not in $\Lex(\Delta^{2}h )S$   will be:
\[xz^{d-1}, yz^{d-1}, z^{d}.\]
To obtain the ideal  $\mathcal{W}_{1}(\Delta h)$ we need to add to $\Lex(\Delta^{2}h )S$ the first two for  $d=3$ and the third for $d=4$. So  we get:
\[\mathcal{W}_{1}(\Delta h) = (x^{2}, xy, y^{2}, xz^{2}, yz^{2}, z^{4}).\]
Now the next step is considering the ideal $\mathcal{W}_{1}(\Delta h) R$, which will have \HF:
\[1, ~4,~ 7,~ 8,~ 8, ~8, \ldots .\] 
The  monomials of $R$ of degree $d > 3$ that are not in $\mathcal{W}_{1}(\Delta h )R$   will be:
\[z^{3}t^{d-3} , xzt^{d-2} , yzt^{d-2} , z^{2}t^{d-2} , xt^{d-1} , yt^{d-1} , zt^{d-1} , t^{d}.\]
So in order to obtain $\mathcal{W}_{2}(h)$ we need to add the first one for $d=4$, the next three for $d=5$ the next three for $d=6$ and the last one for $d=7$. So we get that 
\[ \mathcal{W}_{2}(h)=  (x^{2}, xy, y^{2}, xz^{2}, yz^{2}, z^{4}, z^{3}t , xzt^{3} , yzt^{3} , z^{2}t^{3} , xt^{5} , yt^{5} , zt^{5} , t^{7}).\]

 To construct  the ideal $\mathcal{W}_{1}(h)$ we start directly with the lex-segment ideal for $\Delta h$  in $S = K[x,y,z]$:
\begin{displaymath}  
\Lex(\Delta h) = ( x^{2}, xy, xz, y^{3}, y^{2}z, yz^{2}, z^{4}) 
\end{displaymath}
The ring   $R / (\Lex(\Delta h)R)$ will have again the Hilbert function 
\[1, ~4, ~7,~ 8,~ 8,~8,~ \ldots \]
but  the  monomials of $R$ of degree $d > 3$ that are not in $\Lex(\Delta h )R$ will be this time:
\begin{displaymath}
z^{3}t^{d-3} , y^{2}t^{d-2} , yzt^{d-2} , z^{2}t^{d-2} , xt^{d-1} , yt^{d-1} , zt^{d-1} , t^{d}
\end{displaymath}
And by adding  to $\Lex(\Delta h)R$ in the first monomial for $d=4$, the next three for $d=5$ etc. we obtain: 
\begin{displaymath} 
\mathcal{W}_{1}(h) = ( x^{2}, xy, xz, y^{3}, y^{2}z, yz^{2}, z^{4}, z^{3}t , y^{2}t^{3} , yzt^{3} , z^{2}t^{3} , xt^{5} , yt^{5} , zt^{5} , t^{7}).
\end{displaymath}

We will now give an example of a particular distraction and see how it acts on $\mathcal{W}_{2}(h)$. It is easy to check that the first three lines of the following matrix form a radical distraction for  the ideal $\mathcal{W}_1(\Delta h)$ (as needed by Proposition 5.7):
\begin{displaymath}
\mathcal{L} = 
\left(
\begin{array}{cccccc}
x & x-t & x-2t & x-3t & x-3t &\ldots\\
y & y-t & y-2t & y-3t & y-3t &\ldots\\
z & z-t & z-2t & z-3t & z-3t &\ldots\\
t&t&t&t&t&\ldots
\end{array}
\right)
\end{displaymath}

As the highest degree of the generators in first three variables is $4$, we can consider $L_{ij} = L_{i4}, \forall ~ j \ge 4$. The ideal $D_{\mathcal{L}}( \mathcal{W}_{2}(h))_{\le  3}$ will be:
\begin{displaymath}
D_{\mathcal{L}}( \mathcal{W}_{2}(h))_{\le 3} = (x(x-t),~ xy,~ y(y-t),~ xz(z-t)).
\end{displaymath} 
 One can check easily  that this ideal is radical.
 
 Let us consider also the following ideal:
\begin{displaymath}
 I = (x^{2}, y^{2}, z^{2}, xyzt, xyt^{3}, xzt^{3} , yzt^{3} , xt^{5}, yt^{5}, zt^{5}, t^{7} )
 \end{displaymath}
 By Lemma 3.4 we see immediately that R/I has the WLP. 
 
 In order to have a more general picture of the Betti numbers of algebras with \HF~ $h$, we will also look at $R/\Lex(h)$. This algebra will have the highest Betti numbers possible in this case.

$ \rule{0pt}{3ex}\Lex(h) = (x^2,~xy,~xz,~xt^2,~y^3,~y^2z,~y^2t^2,~yz^3,~yz^2t,~yzt^3, yt^4,~z^5,~z^4t, $

$~z^3t^3,~z^2t^4,~zt^5,~t^7). $
  
\rule{0pt}{3ex}Now let's take a look at the Betti diagrams of the ideals constructed so far (the Betti diagram of $D_{\mathcal{L}}( \mathcal{W}_{2}(h))$ is equal to the one of $\mathcal{W}_{2}(h)$).

\begin{center}
  \begin{tabular}{@{} |c|cccc|c|c|cccc| @{}}
    \cline{1-5}
    \cline{7-11}
    ¥ & 1 & 2 & 3 & 4 & ¥~~~ & ¥ & 1 & 2 & 3 & 4 \\ 
\cline{1-5}
    \cline{7-11}
    1 & 3 & 3   & 1   & - & ¥ ~~~~ & 1 &     3 & 3 & 1 & - \\         
    2 & 3 & 6   & 4   & 2 & ¥ ~~~~ & 2&    3 & 5 & 2 & -  \\ 
    3 & 3 & 8   & 7   & 2 & ¥ ~~~~ & 3&   2 & 5 & 4 & 1 \\ 
    4 & 4 & 11 & 10 & 3 & ¥ ~~~~ & 4 &    3 & 9 & 9 & 3\\ 
    5 & 3 & 9   & 9   & 3 & ¥ ~~~~ & 5 &    3 & 9 & 9 & 3\\ 
    6 & 1 & 3   & 3   & 1 & ¥ ~~~~ & 6&      1 & 3 & 3 & 1 \\ 
  \cline{1-5}
    \cline{7-11}
\multicolumn{5}{c}{\rule{0pt}{3ex}$R/\Lex(h)$} &\multicolumn{1}{c}{~}& \multicolumn{5}{c}{$R/\mathcal{W}_{1}(h)$}\\

\multicolumn{5}{c}{\rule{0pt}{3ex}~} &\multicolumn{1}{c}{~}& \multicolumn{5}{c}{~}\\

   \cline{1-5}
 \cline{7-11}

   ¥ & 1 & 2 & 3 & 4& ¥ &  ¥ & 1 & 2 & 3 & 4 \\ 
 \cline{1-5}
 \cline{7-11}

       1 & 3 & 2 & - & -   &¥ ~~~~ & 1 & 3 & - & - & -\\ 
        2 & 2 & 4 & 2 & -  &¥ ~~~~ & 2 & - & 3 & - & - \\ 
     3 & 2 & 5 & 4 & 1 &¥ ~~~~ & 3 & 1 & 3 & 4 & 1 \\ 
      4 & 3 & 9 & 9 & 3 &¥ ~~~~ & 4 & 3 & 9 & 9 & 3 \\ 
    5 & 3 & 9 & 9 & 3 &¥ ~~~~ & 5 & 3 & 9 & 9 & 3 \\ 
      6 & 1 & 3 & 3 & 1 &¥ ~~~~ & 6 & 1 & 3 & 3 & 1  \\ 
\cline{1-5}
 \cline{7-11}
\multicolumn{4}{c}{}{\rule{0pt}{3ex}$R/\mathcal{W}_{2}(h)$} &\multicolumn{2}{c}{~}& \multicolumn{5}{c}{$R/I$}\\
  \end{tabular}
\end{center}   
  
 

 
 We can notice that $R/\Lex(h)$ has  the largest Betti numbers. Just as predicted, $R/\mathcal{W}_{1}(h)$ has larger Betti numbers than $R/\mathcal{W}_{2}(h)$and $R/I$. We can also notice that the inequality is strict in some cases. The fact that the Betti numbers of $R/\mathcal{W}_{2}(h)$ are all larger than the ones of $R/I$ is just a coincidence.

D{\small IPARTIMENTO DI}  M{\small ATEMATICA,} U{\small NIVERSITA` DI} G{\small ENOVA,}
 G{\small ENOVA,} I-16146  I{\small TALY}\\
\emph{E-mail address:} \texttt{constant@dima.unige.it}

\end{document}